\documentclass{birkjour}
\usepackage{amsmath}
\usepackage{amssymb}
\usepackage{amsthm}
\usepackage{eqlist}
\usepackage{enumerate}
\usepackage[mathscr]{eucal}
\usepackage{hyperref}
\hypersetup{
     colorlinks   = true,
     citecolor    = blue
}
\hypersetup{linkcolor=blue}

 \newtheorem{thm}{Theorem}[section]
 \newtheorem{cor}[thm]{Corollary}
 
 \newtheorem{prop}[thm]{Proposition}
 \theoremstyle{definition}
 \newtheorem{defn}[thm]{Definition}
 \theoremstyle{remark}
 \newtheorem{rem}[thm]{Remark}

 \numberwithin{equation}{section}

\newcommand{\RR}{\mathbb{R}}

\begin{document}

%
%
%
%
%
%
%
%

\title[Harmonic maps and para-Sasakian geometry]
{Harmonic maps and para-Sasakian geometry}

\author[S. K. Srivastava] {S. K. Srivastava}
\address{Department of Mathematics\br
                   Central University of Himachal Pradesh\br
                   Dharamshala-176215\br
                   Himachal Pradesh\br
	       INDIA}
\email{sachin@cuhimachal.ac.in}
\author{K. Srivastava}
\address{Department of Mathematics\\
                    D.D.U. Gorakhpur University\\
                    Gorakhpur-273009\\
                    Uttar Pradesh\\
                    INDIA}
\email{ksriddu22@gmail.com}
 
\thanks {S. K. Srivastava: partially supported through the UGC-BSR Start-Up-Grant vide
their letter no. F.30-29/2014(BSR). K. Srivastava: supported by the Department of Science and Technology, India through the JRF [IF140491] DST/INSPIRE/03/2014/001552.} 

\begin{abstract}
The purpose of this paper is to study the harmonicity of maps to or from para-Sasakian manifolds. We derive the condition for the tension field of paraholomorphic map between almost para-Hermitian manifold and para-Sasakian manifold. The necessary and sufficient condition for a paraholomorphic map between para-Sasakian manifolds to be parapluriharmonic are shown and a non-trivial example is presented for its illustrations. 
\end{abstract}


\maketitle


\section{Introduction}
The study of harmonic maps was initiated by F. B. Fuller, J. Nash and J. H. Sampson \cite{EL,MS} while the first general result on the existence of harmonic maps is due to Eells-Sampson \cite{eells}. Harmonic maps are extrema (critical points) of the energy functional defined on the space of smooth maps between Riemannian (pseudo-Riemannian) manifolds. The trace of the second fundamental form of such maps vanishes. 

More precisely, let $(M_{i},g_{i})$, $i\in\{1,2\}$ be pseudo-Riemannian manifolds and $\Gamma (TM_{i})$ denotes the sections of the tangent bundle $TM_{i}$ of $M_i$,   that is, the space of vector fields on $M_i$. Then {\it energy} $E(f)$ of a smooth map $f:(M_{1},g_{1})\to (M_{2},g_{2})$ is defined by the formula
\begin{align}
E(f)=\int_{M_{1}} e(f)\mathcal{V}_{g_{1}},
\end{align}
where $\mathcal{V}_{g_{1}}$ is the volume measure associated to the metric $g_{1}$ and the {\it energy density} $e(f)$ of $f$ is the smooth function $e(f):M_{1}\to [0,\infty)$ given by
\begin{align}\label{efp}
e(f)_{p}=\frac{1}{2} {\|{f_{*}\|}}^{2}=\frac{1}{2}Tr_{g_{1}}(f^{*}g_{2})(p),
\end{align}
for each $p\in M_{1}$. In the above equation $f_{*}$ is a linear map $f_{*}:\Gamma(TM_{1})\to \Gamma_{f}(TM_{2})$ therefore it can be considered as a section of the bundle $$TM_{1}\otimes f^{-1}(TM_{2})\to M_{1},$$ where $f^{-1}(TM_{2})$ is the pullback bundle having fibres $(f^{-1}(TM_{2}))_{p}=T_{f(p)}M_{2}$, $p\in M_{1}$ and $f^{*}g_{2}$ is the pullback metric on $M_1$. If we denote by $\nabla$ and $\overline{\nabla}$ the Levi-Civita connections on $M_{1}$ and $M_{2}$ respectively, then the second fundamental form of $f$ is the symmetric map $\alpha_{f}:\Gamma(TM_{1})\times \Gamma(TM_{2})\to \Gamma_{f}(TM_{2})$ defined by
\begin{align}\label{ap1}
\alpha_{f}(X,Y)={\widetilde{\nabla}}_{X}f_{*}Y-f_{*}\nabla_{X}Y,
\end{align}
for any $X,\,Y\in \Gamma(TM_{1})$. Where $\widetilde{\nabla}$ is the pullback of the Levi-Civita connection $\overline{\nabla}$ of $M_{2}$ to the induced vector bundle $f^{-1}(TM_{2}):\widetilde{\nabla}_{X}f_{*}Y=\overline{\nabla}_{f_{*}X}f_{*}Y$. The section $\tau(f)\in \Gamma(f^{-1}(TM_{2}))$, defined by 
\begin{align}\label{tauf}
\tau(f)=Tr_{g_{1}}\alpha_{f}\end{align} 
is called the {\it tension field} of $f$ and a map is said to be harmonic if its tension field vanishes identically (see \cite{CI,kupeli}).

If we consider $\{f_{s,t}\}_{s,t\in (-\epsilon,\,\epsilon)}$ a smooth two-parameter variation of $f$ such that $f_{0,0}=f$ and let $V,\,W\in \Gamma(f^{-1}(TM_{2}))$ be the corresponding variational vector fields then
\begin{align*}
V=\frac{\partial}{\partial s}\bigg(f_{s,t}\bigg)\bigg|_{(s,t)=(0,0)},\,\,\,\,\,\,\,W=\frac{\partial}{\partial t}\bigg(f_{s,t}\bigg)\bigg|_{(s,t)=(0,0)}.
\end{align*}
The {\it{Hessian}} of a harmonic map $f$ is defined by:
\begin{align*}
H_{f}(V,W)=\frac{\partial^{2}}{\partial s \partial t}\bigg(E(f_{s,t})\bigg)\bigg|_{(s,t)=(0,0)}.
\end{align*}

The index of a harmonic map $f:(M_{1},g_{1})\to(M_{2},g_{2})$ is defined as the dimension of the tangent subspace of $\Gamma(f^{-1}(TM_{2}))$ on which the Hessian $H_{f}$ is negative definite. A harmonic map $f$ is said to be {\it stable} if Morse index ({\emph{i.e.}}, the dimension of largest subspace of $\Gamma(f^{-1}(TM_2))$ on which the Hessian $H_{f}$ is negative definite) of $f$ is zero and otherwise, it is said to be {\it unstable} (see \cite{gherghe,ianus}). For a non-degenerate point $p\in M_1$, we decompose the space $T_{p}M_{1}$ into its {\it vertical space} $\nu_{p}=\ker f_{*p}$ and its {\it horizontal space} $\mathcal{H}_{p}=(\ker f_{*p})^{\perp}$, that is, $\mathcal{H}_{p}=\nu_{p}^{\perp}$, so that $T_{p}M_{1}=\nu_{p}\oplus \mathcal{H}_{p}$. The map is said to be horizontally conformal if for each $p\in M_{1}$ either the rank of $f_{*p}$ is zero (that is, $p$ is a critical point), or the restriction of $f_{*p}$ to the horizontal space $\mathcal{H}_{p}$ is surjective and conformal (here $p$ is a regular point) \cite{chinea, kupeli}.

The premise of harmonic maps has acknowledged several important contributions and has been successfully applied in computational fluid dynamics (CFD), minimal surface theory, string theory, gravity and quantum field theory  (see \cite{cai, KJJ, MT, NS}). Most of works on harmonic maps are between Riemannian manifolds \cite{hmb}. The harmonic maps between pseudo-Riemannian manifolds behave differently and their study must be subject to some restricted classes of pseudo-Riemannian manifolds \cite{klh}.

This paper is organized as follows. In Sect. \ref{prem}, the basic definitions about almost para-Hermitian manifolds, almost paracontact manifolds and normal almost paracontact manifolds are given. In Sect. \ref{paraholo}, we define and study paraholomorphic map. We prove that the tension field of any $(J,\phi)$-paraholomorphic map between almost para-Hermitian manifold and para-Sasakian manifold lies in $\Gamma(D_1)$. Sect. \ref{parap} deals with parapluriharmonic map in which we obtain the necessary and sufficient condition for a $(\phi_{1},\phi_{2})$-paraholomorphic map between para-Sasakian manifolds to be $\phi_{1}$-parapluriharmonic and give an example for its illustrations. 
\section{Preliminaries}\label{prem}
\subsection{Almost para-complex manifolds}
A smooth manifold $N^{2m}$ of dimension $2m$ is said to be an almost product structure if it admits a tensor field $J$ of type $(1,1)$ satisfying:
\begin{align}\label{j2}
J^2 =Id.
\end{align}
In this case the pair $(N^{2m},J)$ is called an almost product manifold. An almost para-complex manifold is an almost product manifold $(N^{2m}, J)$ such that the eigenbundles $T^{\pm}N^{2m}$ associated with the eigenvalues $\pm 1$ of tensor field $J$ have the same rank \cite{gund}. An almost para-Hermitian manifold $N^{2m}(J,h)$ is a smooth manifold endowed with an almost para-complex structure $J$ and a pseudo-Riemannian metric $h$ compatible in the sense that
\begin{align}\label{gjx}
h(JX,Y)=-h(X, JY),\qquad \forall\,\,X,\,Y\in\Gamma\big(TN^{2m}\big).
\end{align}  
It follows that the metric $h$ has signature $(m,m)$ and the eigenbundles $T^{\pm}N^{2m}$ are totally isotropic with respect to $h$. Let $\{e'_1 , \cdots , e'_m , e'_{m+1}=Je'_{1},\cdots, e'_{2m}=Je'_{m}\}$ be an orthonormal basis and denote $\epsilon'_{i}=g(e'_{i},e'_{i})=\pm 1$:  $\epsilon'_{i}=1$ for $i=1,\cdots, m$ and $\epsilon'_{i}=-1$ for $i=m+1,\cdots, 2m$. The fundamental $2$-form of almost para-Hermitian manifold is defined by
\begin{align}\label{capphi}
\Phi(X,Y)=h(JX,Y)
\end{align}
and the co-differential $\delta\Phi$ of $\Phi$ is given as follows
\begin{align}
(\delta\Phi)(X)=\sum_{i=1}^{2m}\epsilon'_{i}\left(\nabla_{e'_i}\Phi\right)(e'_i , X).
\end{align}
\noindent An almost para-Hermitian manifold is called para-K\"{a}hler if $\nabla J=0 $ \cite{gund}.
\subsection{Almost paracontact metric manifolds}
A $C^{\infty}$ smooth manifold $M^{2n+1}$ of dimension $(2n+1)$ is said to have a triplet $(\phi,\xi,\eta)$-structure if it admits an endomorphism $\phi$, a unique vector field $\xi$  and a $1$-form $\eta$ satisfying: 
\begin{align}\label{eta}
\phi^2& = Id-\eta\otimes\xi \quad {\rm\and}\quad \eta\left(\xi\right)=1,
\end{align}
where $Id$ is the identity transformation; and the endomorphism $\phi$ induces an almost paracomplex structure on each fibre of $\ker\eta,$ the contact subbundle, {\emph{i.e.}}, eigen distributions  ${(\ker\eta)}^{\pm 1}$ corresponding to the characteristic values $\pm 1$ of $\phi$ have equal dimension $n$.\\
From the equation (\ref{eta}), it can be easily deduced that 
\begin{align}\label{phixi}
\phi\xi = 0, \quad \eta\circ\phi = 0 \quad {\rm and \quad rank}(\phi)=2n.
\end{align}
This triplet structure $(\phi,\xi,\eta)$ is called an almost paracontact structure and the manifold $M^{2n+1}$ equipped with the $(\phi,\xi,\eta)$-structure is called an almost paracontact manifold (see also \cite{skm,gn,ks1}). If an almost paracontact manifold admits a pseudo-Riemannian metric $g$ satisfying: 
\begin{align}\label{gphi}
g(\phi X,\phi Y)=-g(X,Y)+\eta(X)\eta(Y),
\end{align}
where signature of $g$ is necessarily $(n+1,\,n)$ for any vector fields $X$ and $Y$; then the quadruple $(\phi,\xi,\eta,g)$ is called an almost paracontact metric structure and the manifold $M^{2n+1}$ equipped with paracontact metric structure is called an almost paracontact metric manifold. With respect to $g$, $\eta$ is metrically dual to $\xi$, that is
\begin{align}\label{gx}
g(X,\xi)=\eta(X). 
\end{align}
Also, equation (\ref{gphi}) implies that
\begin{align}\label{gphix}
g(\phi X,Y)=-g(X,\phi Y).
\end{align}
Further, in addition to the above properties, if the structure-$(\phi,\xi,\eta,g)$ satisfies:
\begin{align}
d\eta(X,Y)=g(X,\phi Y), \nonumber
\end{align} 
for all vector fields $X,\,Y$ on $M^{2n+1}$, then the manifold is called a paracontact metric manifold and the corresponding structure-$(\phi,\xi,\eta,g)$ is called a paracontact structure with the associated metric $g$ \cite{sz}. For an almost paracontact metric manifold, there always exists a special kind of local pseudo-orthonormal basis $\{X_{i}, X_{i^*}, \xi\}$;  where $X_{i^*}=\phi X_{i};$ $\xi$ and $X_{i}$'s are space-like vector fields and $X_{i^*}$'s  are time-like. Such a basis is called a $\phi$-basis. Hence, an almost paracontact metric manifold $M^{2n+1}(\phi,\xi,\eta,g)$ is an odd dimensional manifold with a structure group $\mathbb{U}(n,\mathbb{R})\times Id$, where $\mathbb{U}(n,\mathbb{R})$ is the para-unitary group isomorphic to $\mathbb{G}\mathbb{L}(n,\mathbb{R})$.

\noindent An almost paracontact metric structure-$(\phi,\xi,\eta,g)$ is para-Sasakian if and
only if 	
\begin{align}\label{paras}
(\nabla_{X}\phi) Y=-g(X,Y)\xi+\eta(Y)X.
\end{align}
From Eqs. \eqref{phixi}, \eqref{gphix} and \eqref{paras}, it can be easily deduced for a para-Sasakian manifold that
\begin{equation}\label{nablax}
\nabla_{X}\xi = -\phi X,\quad \nabla_{\xi}\xi=0.
\end{equation}
In particular, a para-Sasakian manifold is $K$-paracontact \cite{sz}.

\subsection{Normal almost paracontact metric manifolds}  
On an almost paracontact metric manifold, one defines the $(1,2)$-tensor field $N_\phi$ by
\begin{align}\label{n}
N_{\phi}:=[\phi,\phi]-2\,d\eta\otimes\xi,
\end{align}
where $[\phi,\phi]$ is the Nijenhuis torsion of $\phi$. If $N_\phi$ vanishes identically, then we say that the manifold $M^{2n+1}$ is a normal almost paracontact metric manifold \cite{sk, sz}. The normality condition implies that the almost paracomplex structure $J$ defined on $M^{2n+1}\times\mathbb{R}$ by
\begin{align}
J\Bigg(X,\lambda\frac{d}{dt}\Bigg)=\Bigg(\phi X+\lambda\xi,\eta(X)\frac{d}{dt}\Bigg)\nonumber
\end{align}
is integrable. Here $X$ is tangent to $M^{2n+1}$, $t$ is the coordinate on $\mathbb{R}$ and $\lambda$ is a $C^\infty$ function on $M^{2n+1}\times\mathbb{R}$.   
Now we recall the following proposition which characterized the normality of almost paracontact metric $3$-manifolds:
\begin{prop}\cite{jwleg}
For an almost paracontact metric $3$-manifold $M^3$, the following three conditions are mutually equivalent
\begin{itemize}
\item[(i)] $M^3$ is normal,
\item[(ii)] there exist smooth functions $p,\,q$ on $M^3$ such that
\begin{align}\label{nablaxphiy}
(\nabla_{X}\phi)Y=q(g(X,Y)\xi-\eta(Y)X)+p(g(\phi X,Y)\xi-\eta(Y)\phi X),
\end{align}
\item[(iii)] there exist smooth functions $p,\,q$ on $M^3$ such that
\begin{align}\label{nablaxxi}
\nabla_{X}\xi=p(X-\eta(X)\xi)+q\phi X,
\end{align}
\end{itemize}
where $\nabla$ is the Levi-Civita connection of the pseudo-Riemannian metric $g$. 
\end{prop}
\noindent The functions $p,\,q$ appearing in Eqs. (\ref{nablaxphiy}) and (\ref{nablaxxi}) are given by    
\begin{align}\label{2alpha}
2p=trace\left\{X\rightarrow\nabla_{X}\xi\right\}, \quad 2q=trace\left\{X\rightarrow\phi\nabla_{X}\xi\right\}.
\end{align}
\noindent A normal almost paracontact metric $3$-manifold is called paracosymplectic if $p=q=0$ and para-Sasakian if $p=0,q=-1$ \cite{ks1}. 
\section{Paraholomorphic map}\label{paraholo}
One can look structure preserving mapping between almost para-Hermitian and almost paracontact manifolds as analogous of the well-known holomorphic mappings in complex geometry \cite{boeckx, calin}.\\
\begin{defn} Let $M_{i}^{2n_{i}+1}(\phi_{i},\xi_{i},\eta_{i},g_{i})$, $i\in\{1,2\}$ be almost paracontact metric manifolds and $N^{2m}(J,h)$ be an almost para-Hermitian manifold. Then a smooth map   
\begin{itemize}
\item[1.] {$f:M_{1}^{2n_{1}+1}\to N^{2m}$ is $(\phi_{1},J)$-paraholomorphic map if $f_{*}\circ \phi_{1}=J\circ f_{*}$. For such a map $f_{*}\xi_{1}=0$.}
\item[2.] {$f:N^{2m}\to M_{1}^{2n_{1}+1}$ is $(J,\phi_{1})$-paraholomorphic map if $f_{*}\circ J=\phi_{1}\circ f_{*}$. Here ${\rm Im} f_{*}\perp \xi_{1}$.}
\item[3.] {$f:M_{1}^{2n_{1}+1}\to M_{2}^{2n_{2}+1}$ is $(\phi_{1},\phi_{2})$-paraholomorphic map if $f_{*}\circ \phi_{1}=\phi_{2}\circ f_{*}$. In particular, $f_{*}(\xi_{1}^{\perp})\subset \xi_{2}^{\perp}$ and $f_{*}(\xi_{1})\sim \xi_{2}$.}
\end{itemize}
When $f_{*}$ interwines the structures upto a minus sign, we say about $(\phi_{1},J)$-anti paraholomorphic, $(J,\phi_{1})$-anti paraholomorphic and $(\phi_{1},\phi_{2})$-anti paraholomorphic mappings.
\end{defn}
\noindent
Now, we prove the following result.
\begin{prop}
Let $f$ be a smooth $(\phi_{1},\phi_{2})$-paraholomorphic map between para-Sasakian manifolds $M_{i}^{2n_i +1}(\phi_{i},\xi_{i},\eta_{i},g_{i})$, $i\in\{1,2\}$. Then
\begin{align}\label{p2}
\phi_{2}(\tau(f))=f_{*}({\rm div}\phi_{1})-Tr_{g_{1}}\beta,
\end{align}
where $\beta(X,Y)=\big(\widetilde{\nabla}_{X}\phi_{2}\big)(f_{*}Y),\qquad \forall\,\,X,\,Y\in\Gamma\big(TM_{1}^{2n_1 +1}\big)$.
\end{prop}
\begin{proof}
Since $f_{*}$ has values in $f^{-1}\big(TM_{2}^{2n_2 +1}\big)$ so that $f_{*}\circ \phi_{1}$ and $\phi_{2}\circ f_{*}$ have values in $f^{-1}\big(TM_{2}^{2n_2 +1}\big)$. Thus, we have
\begin{align}\label{wd}
\big(\widetilde{\nabla}(f_{*}\circ \phi_{1})\big)(X,Y)&=\widetilde{\nabla}_{X}f_{*}(\phi_{1}Y)-(f_{*}\circ \phi_{1})(\nabla_{X}Y) \nonumber
\\
&=\big(\widetilde{\nabla}_{X}f_{*}\big)(\phi_{1}Y)+f_{*}(\nabla_{X}\phi_{1}Y)-(f_{*}\circ \phi_{1})(\nabla_{X}Y) \nonumber
\\
&=\alpha_{f}(X,\phi_{1}Y)+f_{*}((\nabla \phi_{1})(X,Y)).
\end{align}
In the last equality, we have used \eqref{ap1}. On the other hand, we obtain
\begin{align}\label{wde}
\big(\widetilde{\nabla}(\phi_{2}\circ f_{*})\big)(X,Y)&=\widetilde{\nabla}_{X}\phi_{2}(f_{*}Y)-(\phi_{2}\circ f_{*})(\nabla_{X}Y) \nonumber
\\
&=\big(\widetilde{\nabla}_{X}\phi_{2}\big)(f_{*}Y)+\phi_{2}\big(\widetilde{\nabla}_{X}f_{*}Y\big)-\phi_{2}(f_{*}(\nabla_{X}Y)) \nonumber
\\
&=\big(\widetilde{\nabla}_{X}\phi_{2}\big)(f_{*}Y)+\phi_{2}(\alpha_{f}(X,Y)).
\end{align}
From Eqs. \eqref{wd} and \eqref{wde}, we have
\begin{align}\label{phi2a}
\phi_{2}(\alpha_{f}(X,Y))+\big(\widetilde{\nabla}_{X}\phi_{2}\big)(f_{*}Y)=f_{*}((\nabla \phi_{1})(X,Y))+\alpha_{f}(X,\phi_{1}Y).
\end{align}
Let $\{e_{1},e_{2},\cdots,e_{n_1},\phi_{1}e_{1},\phi_{1} e_{2},\cdots,\phi_{1}e_{n_1},\xi_{1}\}$ be a local orthonormal frame for \linebreak $TM_{1}^{2n_1 +1}$. Taking the trace in \eqref{phi2a} and using the fact that $\alpha_{f}$ is symmetric, we have \eqref{p2}. This completes the proof.
\end{proof}
\noindent
Following the proof of the above proposition, we can give the following remarks:
\begin{rem}
For a para-Sasakian manifold $M_{1}^{2n_{1} +1}(\phi_{1},\xi_{1},\eta_{1},g_{1})$ and a para-\linebreak Hermitian manifold $N^{2m}(J,h)$. If  
\begin{itemize}
\item[(a)] $f:M_{1}^{2n_{1}+1}\to N^{2m}$ be a $(\phi_{1},J)$-paraholomorphic map then we have
\begin{align}\label{jtau}
J(\tau(f))=f_{*}{\rm div}\phi_{1}-Tr_{g_{1}}\beta',
\end{align}
where $\beta'(X,Y)=\big(\widetilde{\nabla}_{X}J\big)(f_{*}Y),\qquad \forall\,\,X,\,Y\in\Gamma\big(TM_{1}^{2n_{1}+1}\big)$.\\
\item[(b)] $f:N^{2m}\to M_{1}^{2n_{1}+1}$ be a $(J,\phi_{1})$-paraholomorphic map then we have
\begin{align}\label{ptau}
\phi_{1}(\tau(f))=f_{*}{\rm div}J-Tr_{h}\beta'',
\end{align}
where $\beta''(X,Y)=\big(\widetilde{\nabla}_{X}\phi_{1}\big)(f_{*}Y),\qquad \forall\,\,X,\,Y\in\Gamma\big(TN^{2m}\big)$.\end{itemize}
\end{rem}
\begin{thm}\label{t1}
Let $f$ be a $(\phi_{1},J)$-paraholomorphic map between para-Sasakian manifold $M_{1}^{2n_{1}+1}(\phi_{1},\xi_{1},\eta_{1},g_{1})$ and para-K$\ddot{a}$hler manifold $N^{2m}(J,h)$. Then $f$ is harmonic.
\end{thm}
\begin{proof}
Let $\{e_{1},\cdots,e_{n_1},\phi_{1} e_{1},\cdots,\phi_{1} e_{n_1},\xi_1\}$ be a local orthonormal adapted basis on $TM_{1}^{2n_1+1}$, then from Eqs. \eqref{phixi} and \eqref{paras}, we have ${\rm div}\phi_{1}=0$ (since for a $(\phi_{1},J)$-paraholomorphic map $f_{*}\xi_{1}=0$). It follows by the use of equation \eqref{jtau} that $J(\tau(f))=0$ as $N^{2m}$ is a para-K$\ddot{\rm a}$hler manifold. Therefore, $\tau(f)=0$ and $f$ is harmonic. This completes the proof of the theorem.
\end{proof}

\noindent For $i\in\{1,2\}$, let $D_{i}$ be real distributions, respectively, on para-Sasakian manifolds $M_{i}^{2n_i+1}$ of rank $2n_{i}$ then it admits globally defined $1$-form $\eta_{i}$ such that $D_{i}\subseteq \ker\eta_{i}$. Clearly, $TM_{i}^{2n_i+1} =D_{i}\oplus\{\xi_{i}\}$, where $\{\xi_{i}\}$ is the real distribution of rank one defined by $\xi_{i}$ \cite{calin}. 

\noindent Now, we prove:  
\begin{thm}\label{hk1}
For any $(J,\phi_{1})$-paraholomorphic map $f$ between almost para-\linebreak Hermitian manifold $N^{2m}(J,h)$ and para-Sasakian manifold $M_{1}^{2n_1+1}(\phi_{1},\xi_{1},\eta_{1},g_{1})$, the tension field $\tau(f)\in \Gamma(D_{1})$. 
\end{thm}

\noindent Before going to proof of this theorem, we first prove the following proposition:

\begin{prop}
For an almost para-Hermitian manifold $N^{2m}(J,h)$, we have
\begin{align}\label{sni}
\sum^{m}_{i=1}\Big\{\nabla_{Je'_{i}}Je'_{i}-\nabla_{e'_{i}}e'_{i}\Big\}=J\Bigg\{{\rm div}J-\sum^{m}_{i=1}\Big[e'_{i},Je'_{i}\Big]\Bigg\}
\end{align}
where $\{e'_{1},e'_{2},\cdots,e'_{m},Je'_{1},Je'_{2},\cdots,Je'_{m}\}$ is a local orthonormal frame on $TN^{2m}$.
\end{prop}
\begin{proof}
It is straightforward to calculate 
\begin{align}\label{prop1}
{\rm div}J&=\sum^{m}_{i=1}\Big\{[e'_{i},Je'_{i}]-J(\nabla_{e'_{i}}e'_{i})+J(\nabla_{Je'_{i}}Je'_{i})\Big\}
\end{align}
and the result follows from \eqref{j2} and \eqref{prop1}. This completes the proof. 
\end{proof}
\noindent{\it Proof of Theorem \ref{hk1}.}
Since $f_{*}(X)\in \Gamma(D_{1})$, $\forall\,X\in \Gamma\big(TN^{2m}\big)$ therefore for any local orthonormal frame $\{e'_{1},e'_{2},\cdots,e'_{m},Je'_{1},Je'_{2},\cdots,Je'_{m}\}$ on $TN^{2m}$, we obtain by using Eqs. \eqref{ap1}, \eqref{tauf}, \eqref{gphix} and \eqref{nablax} that
\begin{align}\label{gtauf}
g_{1}(\tau(f),\xi_{1})=\sum^{m}_{i=1}\Big\{g_{1}(f_{*}(\nabla_{Je'_{i}}Je'_{i}),\xi_{1})-g_1(f_{*}(\nabla_{e'_{i}}e'_{i}),\xi_{1})\Big\}.
\end{align}
Employing Eq. \eqref{sni}, the above equation reduces to
\begin{align}\label{gtauf1}
g_{1}(\tau(f),\xi_{1})=g_{1}\Bigg(\phi_{1} f_{*}\Bigg({\rm{ div}}J-\sum^{m}_{i=1}J[e'_{i},Je'_{i}]\Bigg),\xi_{1}\Bigg).
\end{align}
Reusing Eq. \eqref{gphix} in \eqref{gtauf1}, we get
\begin{align*}
g_{1}(\tau(f),\xi_{1})=0,
\end{align*}
which shows that $\tau(f)\in \Gamma(D_{1})$. This completes the proof of the theorem.
\qed

\noindent By the consequence of the above theorem we can state the following result as a corollary of the theorem \ref{hk1}.
\begin{cor}
Let $N^{2m}(J,h)$ and $M_{1}^{2n_{1}+1}(\phi_{1},\xi_{1},\eta_{1},g_{1})$ be para-K$\ddot{a}$hler and para-Sasakian manifolds respectively. Then for any $(J,\phi_{1})$-paraholomorphic map $f:N^{2m}\to M_{1}^{2n_{1}+1}$, the tension field $\tau(f)\in \Gamma(D_{1})$.
\end{cor}
\section{Parapluriharmonic map}\label{parap}
In this section we define the notion of $\phi_{1}$-parapluriharmonic map which is similar to the notion of $\phi$-pluriharmonic map between almost contact metric manifold and Riemannian manifold, for $\phi$-pluriharmonic map see : \cite{boeckx,inpas}. 
\begin{defn}
A smooth map $f$ between almost paracontact metric manifold $M_{1}^{2n_1+1}(\phi_{1},\xi_{1},\eta_{1},g_{1})$ and pseudo-Riemannian manifold $N^{m},$ is said to be $\phi_{1}$-\linebreak parapluriharmonic if 
\begin{align}\label{pluri}
\alpha_{f}(X,Y)-\alpha_{f}(\phi_{1} X,\phi_{1}Y)=0,\qquad \forall\,\,X,\,Y\in \Gamma\big(TM_{1}^{2n_1+1}\big),
\end{align}
where the second fundamental form $\alpha_f$ of $f$ is defined by \eqref{ap1}. In particular, $\alpha_{f}(X,\xi_{1})=0$ for any tangent vector $X$.
\end{defn}
\begin{prop}\label{pro1}
Any $\phi_{1}$-parapluriharmonic map $f$ between almost paracontact metric manifold $M_{1}^{2n_1+1}(\phi_{1},\xi_{1},\eta_{1},g_{1})$ and pseudo-Riemannian manifold $N^{m}$ is harmonic.
\end{prop}
\begin{proof}
Let $\{e_{1},\cdots,e_{n_1},\phi_{1} e_{1},\cdots,\phi_{1} e_{n_1},\xi_1\}$ be a local orthonormal frame on $TM_{1}^{2n_1+1}$, then by definition of $\phi_{1}$-parapluriharmonicity, we have
\begin{align*}
\alpha_{f}(\xi_{1},\xi_{1})=0 \quad {\rm{and}} \quad \alpha_{f}(e_{i},e_{i})-\alpha_{f}(\phi_{1} e_{i},\phi_{1} e_{i})=0,
\end{align*}
for $i\in \{1,2,\cdots,n\}$. Therefore, $\tau(f)=Tr_{g_{1}}\alpha_{f}=0$.
This completes the proof.
\end{proof}
\begin{thm}
Let $f$ be a smooth $(\phi_{1},J)$-paraholomorphic map between normal almost paracontact metric 3-manifold $M_{1}^{3}(\phi_{1},\xi_{1},\eta_{1},g_{1})$ and para-K$\ddot{a}$hler manifold $N^{2}(J,h)$. Then $f$ is harmonic. 
\end{thm}
\begin{proof}
We recall that $f_{*}\xi_{1}=0$ for a $(\phi_{1},J)$-paraholomorphic map and $N^2$ is para-K$\ddot{\rm{a}}$hler, and that from Eq. \eqref{nablaxphiy} for any vectors $X, Y$ tangent to $M_{1}^3$, we have  
\begin{align}\label{bla}
f_{*}(\nabla_{X}\phi_{1})Y=-\{qf_{*}X+p f_{*}\phi_{1} X\}\eta_{1}(Y).
\end{align}
Using equation \eqref{phi2a} for a given map, we obtain
\begin{align}\label{ja}
J(\alpha_{f}(X,Y))=-\{q f_{*}X+p f_{*}\phi_{1} X\}\eta_{1}(Y)+\alpha_{f}(X,\phi_{1} Y).
\end{align}
Replacing $Y$ by $\phi_{1} Y$ and employing Eqs. \eqref{eta} and \eqref{phixi}, the above equation reduces to
\begin{align*}
J(\alpha_{f}(X,\phi_{1} Y))=\alpha_{f}(X,Y).
\end{align*}
By the virtue of the fact that $\alpha_{f}$ is symmetric, we obtain from above equation that
\begin{align*}
\alpha_{f}(X,Y)-\alpha_{f}(\phi_{1} X,\phi_{1} Y)=0.
\end{align*}
The above expresion implies that $f$ is $\phi_{1}$-parapluriharmonic and thus harmonic from the proposition \ref{pro1}. This completes the proof of the theorem.
\end{proof}

\noindent As an immediate consequence of above theorem and remark $2.4$ of \cite{ks1} one easily gets the following corollary:
\begin{cor}
Let $M_{1}^{3}(\phi_{1},\xi_{1},\eta_{1},g_{1})$ be a normal almost paracontact metric 3-manifold with $p,q=$constant, $N^{2}(J,h)$ be a para-K$\ddot{a}$hler manifold and $f:M_{1}^{3}\to N^{2}$ be a smooth $(\phi_{1},J)$-paraholomorphic map. Then $M_{1}^{3}$ is paracosymplectic manifold.
\end{cor}
\noindent Here, we derive the necessary and sufficient condition for a $(\phi_{1},\phi_{2})$-paraholomorphic map between para-Sasakian manifolds to be $\phi_{1}$-parapluriharmonic.
\begin{thm}\label{main}
Let $f:M_{1}^{2n_{1}+1}\to M_{2}^{2n_{2}+1}$ be a $(\phi_{1},\phi_{2})$-paraholomorphic map between para-Sasakian manifolds $M_{i}^{2n_{i}+1}(\phi_{i},\xi_{i},\eta_{i},g_{i})$, $i\in\{1,2\}$. Then $f$ is $\phi_{1}$-parapluriharmonic if and only if $\xi_{2}\in\left(Im\,f_{*}\right)^{\perp}$.
\end{thm}
\begin{proof}
Since $f$ is a $(\phi_{1},\phi_{2})$-paraholomorphic map then for all $x\in M_{1}^{2n_{1}+1}$ there exists a function $\lambda$ on $M_{1}^{2n_{1}+1}$ such that 
\begin{align}\label{feta} \left(f_{*}\xi_{1}\right)_{f(x)}=\lambda(x)\left(\xi_{2}\right)_{f(x)}\quad{\rm and}\quad \left(f^{*}\eta_{2}\right)_{x}=\lambda(x)\left(\eta_{1}\right)_{x}.
\end{align}
For any $X,Y\in\Gamma\left(D_{1}\right)$, we have from Eqs. \eqref{ap1}, \eqref{paras} and \eqref{feta} that 
\begin{align*}
\alpha_{f}(X,\phi_{1}Y)=\phi_{2}\alpha_{f}(X,Y)+\eta_{2}(f_{*}X)f_{*}Y-g_{2}(f_{*}X,f_{*}Y)\xi_{2}+\lambda g_{1}(X,Y)\xi_{2}.
\end{align*}
From above equation and the fact that $\alpha_{f}$ is symmetric, we obtain that 
\begin{align}
\alpha_{f}(X,\phi_{1}Y)-\alpha_{f}(\phi_{1}X, Y)=\eta_{2}(f_{*}Y)f_{*}X-\eta_{2}(f_{*}X)f_{*}Y.
\end{align}
Replacing $Y$ by $\phi_{1}Y$ in above expression and using Eqs. \eqref{eta} and \eqref{phixi}, we find
\begin{align}
\alpha_{f}(X,Y)-\alpha_{f}(\phi_{1}X, \phi_{1}Y)=-\eta_{2}(f_{*}X)\phi_{2}(f_{*}Y).
\end{align} 
This implies that $\alpha_{f}(X,Y)-\alpha_{f}(\phi_{1}X, \phi_{1}Y)=0$ if and only if $\xi_{2}\in\left(Im\,f_{*}\right)^{\perp}$.  This completes the proof of the theorem.
\end{proof}
\noindent Now, we present an example for illustrating theorem \ref{main}:
\subsection{Example}
Let $M_{i}^3\subset\RR^3,\ i\in\{1,2\}$ be $3$-dimensional manifolds with standard Cartesian coordinates. Define the almost paracontact structures $\left(\phi_{i},\xi_{i},\eta_{i},g_{i}\right)$ respectively on $M_{i}^3$ by
\begin{align}
&\phi_{1}e_{1}=-e_{2}+x^2e_{3},\ \phi_{1}e_{2}=-e_{1},\ \phi_{1}e_{3}=0,\ \xi_{1}=e_{3},\ \eta_{1}=x^2dy+dz,\\
&\phi_{2}e_{1}'=-e_{2}',\ \phi_{2}e_{2}'=-e_{1}'+v^2e_{3}',\ \phi_{2}e_{3}'=0,\ \xi_{2}=e_{3}',\ \eta_{2}=-v^2du+dw,
\end{align}
where $e_1=\frac{\partial}{\partial x}$, $e_2=\frac{\partial}{\partial y}$, $e_3=\frac{\partial}{\partial z}$, $e_1'=\frac{\partial}{\partial u}$, $e_2'=\frac{\partial}{\partial v}$ and $e_3'=\frac{\partial}{\partial w}$. By direct calculations, one verifies that the Nijenhuis torsion of $\phi_i$ for $i\in\{1,2\}$ vanishes, which implies that the structures are normal. Let the pseudo-Riemannian metrics $g_i,\ i\in\{1,2\}$ are  prescribed respectively on $M_{i}^3$ by
\begin{align}
\left[g_1\left(e_s,e_t\right)\right]=\begin{bmatrix}-x&0&0\\0&x^4+x&x^2\\0&x^2&1\end{bmatrix},\ \left[g_2\left(e_s',e_t'\right)\right]=\begin{bmatrix}v^4+v&0&v^2\\0&-v&0\\v^2&0&1\end{bmatrix},
\end{align}
for all $s,t\in\{1,2,3\}$. For the Levi-Civita connections $\nabla, \overline\nabla$ with respect to metrics $g_1, g_2$ respectively, we obtain
\begin{align*}
&\nabla_{{e_1}}{e_1}=\frac{1}{2x}e_1,\, \nabla_{{e_1}}{e_2}=\frac{2x^3+1}{2x}e_2+\left(\frac{x}{2}-x^4\right){e_{3}}=\nabla_{{e_2}}{e_1},\, \nabla_{{e_2}}{e_2}=\frac{4x^3+1}{2x}e_1,\\ 
&\nabla_{{e_1}}{e_3}={e_2}-x^2{e_3}=\nabla_{{e_3}}{e_1}, \, \nabla_{{e_2}}{e_3}={e_1}=\nabla_{{e_3}}{e_2}, \, \nabla_{{e_3}}{e_3}=0,\end{align*}
\begin{align*}
&\overline\nabla_{{e_1}'}{e_1}'=\frac{4v^3+1}{2v}e_2',\, \overline\nabla_{{e_1}'}{e_2}'=\frac{2v^3+1}{2v}e_1'+\left(\frac{v}{2}-v^4\right){e_{3}}'=\overline\nabla_{{e_2}'}{e_1}',\, \overline\nabla_{{e_3}'}{e_3}'=0,\\ 
&\overline\nabla_{{e_2}'}{e_3}'={e_1}'-v^2e_3'=\overline\nabla_{{e_3}'}{e_2}',\, \overline\nabla_{{e_2}}'{e_2}'=\frac{1}{2v}e_2', \, \overline\nabla_{{e_3}'}{e_1}'={e_2}'=\overline\nabla_{{e_1}'}{e_3}'.  \end{align*}
From above expressions and equation \eqref{nablaxxi}, we find $p=0$, $q=-1$. Hence the $M_{1}^3$ and $M_{2}^3$ are para-Saakian manifolds with invariant distributions $D_1={\rm span}\{{e_1},\phi_{1}e_{1}\}$ and $D_2={\rm span}\{{e_2}',\phi_{2}e_{2}'\}$ respectively. Let $f:M_{1}^3\to M_{2}^3$ be a mapping defined by $f(x,y,z)=(y,x,z)$. Then $f_{*}\circ \phi_{1}=\phi_{2}\circ f_{*}$, {\emph i.e.}, $f$ is a $(\phi_{1},\phi_{2})$-paraholomorphic map between para-Sasakian manifolds. For any $X,Y\in\Gamma({D_1})$ and $x\in M_{1}^3$, it is not hard to see that $\alpha_{f}(X,Y)=\alpha_{f}(\phi_{1}X, \phi_{1}Y)$, $\lambda(x)=1$ and $g_2(\xi_{2},f_{*}X)=0.$ Thus theorem \ref{main} is verified.

\end{document}